\documentclass[3p]{elsarticle}

\usepackage{graphicx}
\usepackage{amsmath,amssymb,amsfonts}
\usepackage{subfigure}
\usepackage{tikz}
\usetikzlibrary{circuits.ee.IEC}

\usepackage{enumitem}
\newlist{hyp}{enumerate}{1}
\setlist[hyp]{label = {\bf(A\arabic*)}, resume}

\newlist{facts}{enumerate}{1}
\setlist[facts]{label = {\bf(F\arabic*)}, resume}

\newlist{samp}{enumerate}{1}
\setlist[samp]{label = {\bf(S\arabic*)}, resume}

\newtheorem{lemma}{Lemma}

\newtheorem{thm}{Theorem}

\newenvironment{proof}{\vspace{-0.05in}\noindent{\sc Proof:}}%
        {\hspace*{\fill}$\Box$\par}
\newenvironment{proofof}[1]{\smallskip\noindent{\sc Proof of #1.}}%
        {\hspace*{\fill}$\Box$\par}

\newdefinition{defn}{Definition}
\newdefinition{exam}{Example}
\newdefinition{rem}{Remark}
\newdefinition{assmpt}{Assumption}
\newdefinition{claim}{Claim}

\newcommand{\N}{\mathbb{N}}

\newcommand{\R}{\mathbb{R}}

\newcommand{\cC}{\mathcal{C}}

\newcommand{\cL}{\mathcal{L}}

\newcommand{\cP}{\mathcal{P}}
\newcommand{\cQ}{\mathcal{Q}}

\DeclareMathOperator{\im}{Im}

\DeclareMathOperator{\col}{col}

\newcommand{\nom}{\text{\normalfont nom}}

\newcommand{\eps}{\varepsilon}

\newcommand{\bdiag}{\text{\rm block diag\,}}

\colorlet{Darkred}{red!50!black}
\colorlet{Darkgreen}{green!50!black}
\xdefinecolor{UIUCblue}{rgb}{0,0.24,0.49} 
\xdefinecolor{UIUCorange}{rgb}{0.96,0.5,0.14} 

\title{Observer-Based Feedback Stabilization of Linear Systems with Event-triggered
Sampling and Dynamic Quantization\tnoteref{t1}}
\tnotetext[t1]{This work has been partially supported by the ANR project LimICoS contract number 12-BS03-005-01.}

\author[laas]{Aneel Tanwani\corref{cor1}}
\ead{aneel.tanwani@laas.fr}
\author[gipsa]{Christophe Prieur}
\ead{christophe.prieur@gipsa-lab.fr}
\author[gipsa]{Mirko Fiacchini}
\ead{miko.fiacchini@gipsa-lab.fr}

\cortext[cor1]{Corresponding author.}
\address[laas]{LAAS (CNRS), 7 Avenue Colonel du Roche, 31400 Toulouse, France.}
\address[gipsa]{Gipsa-Lab (CNRS), 11 Rue des Math\'ematiques, BP 46, 38402 St. Martin d'H\`eres, France.}

\begin{document}

%
%
\begin{frontmatter}
\begin{abstract}
We consider the problem of output feedback stabilization in linear systems when the measured outputs and control inputs are subject to event-triggered sampling and dynamic quantization.
A new sampling algorithm is proposed for outputs which does not lead to accumulation of sampling times and results in asymptotic stabilization of the system.
The approach for output sampling is based on defining an event function that compares the difference between the current output and the most recently transmitted output sample not only with the current value of the output, but also takes into account a certain number of previously transmitted output samples.
This allows us to reconstruct the state using an observer with sample-and-hold measurements.
The estimated states are used to generate a control input, which is subjected to a different event-triggered sampling routine; hence the sampling times of inputs and outputs are asynchronous.
Using Lyapunov-based approach, we prove the asymptotic stabilization of the closed-loop system and show that there exists a minimum inter-sampling time for control inputs and for outputs.
To show that these sampling routines are robust with respect to transmission errors, only the quantized (in space) values of outputs and inputs are transmitted to the controller and the plant, respectively.
A dynamic quantizer is adopted for this purpose, and an algorithm is proposed to update the range and the centre of the quantizer that results in an asymptotically stable closed-loop system.

\end{abstract}


\end{frontmatter}

\section{Introduction}

For sampled data control of continuous-time dynamical systems, event-triggered techniques have regained interest over the past 5-6 years \cite{Astr08}, where the measurements are not sent periodically to the controller, but instead the sampling times are determined based on the current value of the state.
A recent article \cite{HeemJoha12} provides a tutorial exposition into the subject, and sums up most of the work done so far.
A common framework for event-triggered control involves a stabilizing feedback controller and a triggering mechanism that determines when to send the updated measurements to the controller.
While the feedback control is usually available ``off-the-shelf,'' different strategies and variants are adopted for triggering mechanism depending upon the particular problem setup.
Initial approaches for event-triggering mechanism involve keeping the difference between current value of the state and the last updated measurement relatively small \cite{AntaTabu10, LunzLehm10, Tabu07}.
Another technique is to monitor the derivative of the Lyapunov function associated with the closed-loop system, and if it starts approaching zero, then we update the measurement knowing that the new measurement will make the derivative sufficiently negative \cite{MarcDura13, PostAnta11, SeurPrie11}.
The effect of disturbances in plant dynamics could also be taken into account by such methods if the triggering mechanism is modified appropriately \cite{MazoAnta10}.
Moreover, event-triggered control has also been used for stabilization of systems in the presence of  networks \cite{DePeFras13, DePeSail13, MazoCao11, WangLemm11}.

In the references cited above, the triggering mechanisms are based on using the full-state measurements and when it comes to using output (partial state) measurements for feedback, rather than full-state feedback, then relatively little has been done.
If we directly generalize the techniques based on keeping the error (between the last sampled output and the current value of the output) small, then such methods lead to Zeno phenomenon, where we need to send infinitely many samples in finite-time and hence the technique is not feasible for implementation in practice.
Some refinements have been proposed in \cite{DonkHeem12, LehmLunz11, TallChop12}, where instead of asymptotic stabilization, the authors modify the event function to achieve practical stabilization so that the trajectories of the closed-loop system only converge to a ball defined as a design parameter.
Eventually, that parameter also determines the minimum inter-sampling time as well.
Asymptotic stability with output-feedback and event-triggered sampling has also been considered in more recent works where a certain dwell-time is enforced between two consecutive sample updates to overcome Zeno phenomenon.
The so-called periodic event-triggered control could be seen as an implementation of this idea \cite{HeemDonk13a, HeemDonk13b} where it is assumed that the continuous-time plant is already discretized with some fixed sampling-time, or a certain sampling period is precalculated to asymptotically stabilize the system. One then focuses on adding another level of sampling strategy (which is event-triggered) that would reduce the sampling rate for measurements even further. The results appearing in \cite{HeemDonk13b} for linear systems take disturbances into account, and derive minimum inter-sampling time for full-state feedback case only. The idea of forcing a certain dwell-time between two consecutive sample-updates has also been adopted in nonlinear setting for output feedback laws in \cite{AbdePost14}.

In this paper, we propose a dynamic output feedback controller for asymptotic stabilization using event-triggered sampling which does not rely on precalculating some fixed sampling period between output updates.
Our framework involves computing the sampling times for outputs and inputs separately.
Just like the approach adopted in the state-feedback case, our approach is also based on keeping the error between the current value of the output and the last sampled output small.
The crux of our approach is to compare this difference not only with the norm of the current value of the output, but with the norm of a vector that comprises some previously transmitted output measurements.
If we pick a sufficient number of past samples, then these samples contain enough information about the norm of the state (due to the observability assumption).
The controller, using these sampled outputs, is designed based on the principle of {\em certainty equivalence}. An estimate of the current state is first computed, which is in turn fed into the control law.
The control inputs transmitted to the plant are also time-sampled, where the event-triggering rule depends upon the state of the controller.
To show that our sampling algorithms are feasible for implementation, we derive an expression for minimum inter-sampling time between the sampled measurements sent to the controller and the plant.
A property of the proposed sampling routines is that the sampling times of the output and control input are not necessarily synchronized.

As an added practical consideration and to show that our strategy is robust with respect to transmission errors, we assume that the sampled outputs and sampled inputs are subjected to quantization as well, that is, the output and inputs are transmitted to the controller and the plant, respectively using a string of finite alphabets only.
However, to preserve asymptotic stability, the model of the quantizer is assumed to be dynamic as used in \cite{Libe03Aut, Libe03TAC}, that is the parameter that determines the range and the sensitivity of the quantizer can be scaled.
Event-triggered sampling with static quantization is also considered in the works of \cite{GarcAnts13} with full state feedback, and with output under a passivity assumption on the plant dynamics in \cite{YuAnts13}, but none of these works allow the possibility of designing different sampling and quantization algorithms for inputs and outputs.
The novelty in handling the quantization comes from the fact that we are working with an observer-based controller where the outputs and control inputs are both subject to quantization, and the sampling is event-based and not periodic.

To summarize, this paper proposes algorithms for event-triggered sampling and dynamic quantization of input and output measurements of linear time-invariant systems, also see Figure~1.
Such architectures could be useful when the control action is computed on a server located far away from the plant and the communication is carried over some communication channel between the plant and the controller.
The paper proposes algorithm on how and when the information between the plant and the controller must be transmitted, and the contribution could be summed up through following observations:

\begin{itemize}
\item We can achieve asymptotic stabilization using dynamic output feedback and event-triggered sampling of the output measurements without imposing time-regularization or fixed periodic sampling as done in the literature, provided we use the information of previously sampled outputs, and not just the last sampled measurement.
\item The event-triggered sampling algorithms are robust with respect to transmission errors, which in this paper manifest in the form of quantization.
If these errors vanish (which happens due to dynamic quantization) then the state of the system also converges to the origin asymptotically.
\item A trade-off between {\em how fast we sample} compared to {\em how precisely we quantize the measurements} also follows from our results. It appears in the form of design parameters introduced for sampling and quantization, respectively. In particular, when the dynamic parameter for quantization is very large (so that quantized measurements are very coarse), one has to sample quite fast, whereas smaller values of the quantization parameter (more exact measurements) possibly allow for larger inter-sampling times.
\end{itemize}
\section{Problem Setup}
\begin{figure}[!t]
\centering
\begin{tikzpicture}[xscale = 0.9, yscale = 0.8, circuit ee IEC,
every info/.style={font=\footnotesize}]
\draw (-2,0) node [rectangle, rounded corners, draw, minimum height =0.65cm, text centered] (sys) {$ \cP: \left\{ \begin{aligned}\dot x & = Ax + B q_\mu(u_\nom(\tau_j)) \\ y&=Cx\end{aligned}\right .$};
\draw (0.5,-5) node [rectangle, rounded corners, draw, minimum height =0.65cm, text centered] (obs) {\small $\cC: \left\{ \begin{aligned} &\dot z = Az + B q_\mu(u_\nom(\tau_j)) +L q_\nu(y(t_k) - Cz(t_k))\\  &u_\nom(t)= Kz(t)\end{aligned}\right .$};
\coordinate (upSamp) at ([xshift=2.5cm]sys.east);
\coordinate (downSamp) at ([xshift=-1.5cm]obs.west);
\draw (upSamp) node [rectangle, rounded corners, draw, minimum height =0.65cm, text width = 2cm, text centered] (uZOH) {Output  Sampler};
\draw (downSamp) node [rectangle, rounded corners, draw, minimum height =0.65cm, text width = 1.5cm, text centered] (dZOH) {Input Sampler};
\coordinate (tr) at ([xshift=1.5cm]uZOH.east);
\coordinate (br) at ([yshift=-5cm]tr);
\coordinate (bl) at ([xshift=-1cm]dZOH.west);
\coordinate (tl) at ([yshift=5cm]bl);
\coordinate (upr) at ([yshift=-1.25cm]tr);
\coordinate (downr) at ([yshift=-3.5cm]tr);
\coordinate (upl) at ([yshift=-1.25cm]tl);
\coordinate (downl) at ([yshift=-3.5cm]tl);
\draw (upr) node [rectangle, rounded corners, draw, minimum height =0.65cm, text centered] (qOup) {Encoder};
\draw (downr) node [rectangle, rounded corners, draw, minimum height =0.65cm, text centered] (qOdown) {Decoder};
\draw (upl) node [rectangle, rounded corners, draw, minimum height =0.65cm, text centered] (qUup) {Decoder};
\draw (downl) node [rectangle, rounded corners, draw, minimum height =0.65cm, text centered] (qUdown) {Encoder};

\draw [thick, ->] (sys.east) -- (uZOH.west);
\draw[thick] (uZOH.east) node[anchor=south west] {$y(t_k)$} -- (tr);
\draw [thick,->] (tr)--(qOup.north);
\draw [thick,dashed,->] (qOup.south) -- (qOdown.north);
\draw [thick] (qOdown.south)--(br);
\draw [thick,->] (br) -- (obs.east);
\draw [thick, ->] (obs.west) -- (dZOH.east);
\draw[thick] (dZOH.west)node[anchor=north east] {$u_\nom(\tau_j)$} -- (bl);
\draw [thick,->] (bl)--(qUdown.south);
\draw [thick,dashed,->] (qUdown.north) -- (qUup.south);
\draw [thick] (qUup.north)--(tl);
\draw [thick,->] (tl) -- (sys.west);
\end{tikzpicture}
\caption{Feedback loop where the inputs and outputs are time-sampled and quantized.}
\label{fig:loop}
\vskip -1em
\end{figure}
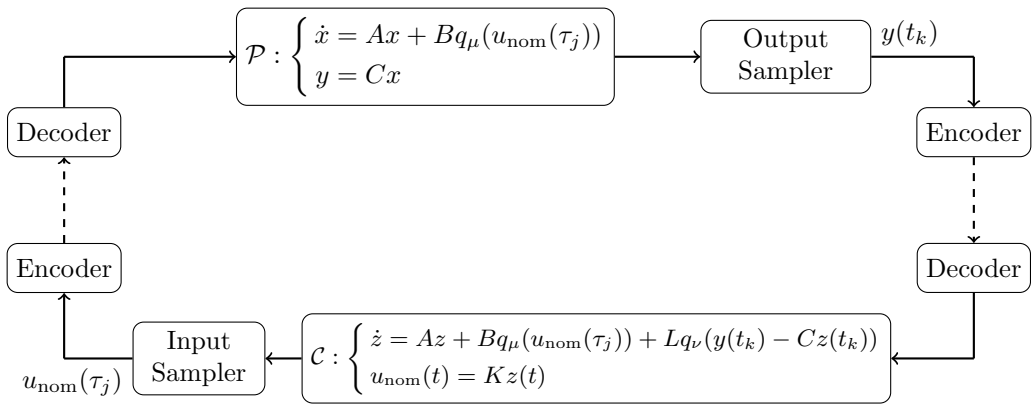

We consider linear time-invariant systems described as:
\begin{equation}\label{eq:sysLin}
\cP: \left \{
\begin{aligned}
\dot x(t) &= Ax(t) +Bu(t),\\ 
y(t) &= Cx(t)
\end{aligned}
\right .
\end{equation}
where $x(t) \in \R^n$ is the state, $y(t) \in \R^p$ is the measured output, and $u(t) \in \R^m$ is the control input.
We are interested in feedback stabilization of the control system~\eqref{eq:sysLin} by realizing the control architecture proposed in Figure~\ref{fig:loop}, which includes in particular the following design elements:
\begin{itemize}[leftmargin=1em]
\item {\bf Asynchronous event-triggered sampling:} The sampling times for outputs and inputs, denoted by $t_k$ and $\tau_j$, $j,k \in \N$, respectively,  are determined based on certain event-triggering strategies.
\item {\bf Dynamic quantization:} An update rule is specified for the design parameter (a discrete variable) of the quantizers.
\end{itemize}

\subsection{Information Processing in Closed-Loop}

Before presenting the design of controller, and the algorithms for sampling and quantization, let us briefly discuss how the information is processed and transmitted in our proposed strategy. It is assumed that there is a processor and an encoder attached next to the sensors measuring the output of the plant. The processor first determines when to send the updated values of the output using a sampling algorithm, and the encoder then determines how to encode the sampled output with finitely many alphabets. Thus, we assume that the exact values of the output are used to determine sampling instants, and then the sampled values are quantized and sent to the controller with time-stamp (because the controller cannot recompute the sampling instants itself).
This (sampled and quantized) information is then passed to a controller, where a decoder first determines the value of the output modulo certain error due to sensitivity of the quantizer. 
This output is then injected into a (continuous-time) dynamic controller which basically computes the control input for stabilization of the plant.
When communicating the inputs to the plant, the same mechanism is used as that on the output side, that is, a processor and an encoder determine when and what to transmit, and the decoder attached to the plant actuators then decodes the symbol to the actual control value (while assuming that the sampling times chosen by the controller are transmitted to the plant as well).



%
\subsection{Controller Structure}
The proposed dynamic controller $\cC$ has the following form:
\begin{subequations}\label{eq:contLin}
\begin{equation}
\dot z(t) = Az(t) + B u(t)+Lq_{\nu}(y(t_k)-Cz(t_k)), \qquad t \in [t_k,t_{k+1}), \label{eq:contLina}\\
\end{equation}
and the input $u$ is the quantized value of the nominal control law $u_\nom=Kz$ sampled at $\tau_j$ , that is,
\begin{equation}
u(t) = q_\mu(Kz(\tau_j)), \quad t \in [\tau_{j}, \tau_{j+1}).
\end{equation}
\end{subequations}
Here $y(t_k)$, and $u(\tau_j)$, $k,j\in \N$, denote the sampled values of output and input, respectively.
The output quantizer $q_\nu : \R^p \rightarrow \cQ_y$, and input quantizer $q_\mu:\R^m \rightarrow \cQ_u$ for some finite sets $\cQ_y$ and $\cQ_u$, include the design parameters $\nu: [t_0,\infty) \rightarrow \R_+$, and $\mu:[\tau_0,\infty) \rightarrow \R_+$ respectively, which are piece-wise constant and are only updated at $t_k$ and $\tau_j$ respectively. For notational convenience, we will often denote $\nu(t_k)$ by $\nu_k$, and $\mu(\tau_j)$ by $\mu_j$.
In writing equation~\eqref{eq:contLin}, it must be noted that the discrete measurements received by the controller have been passed through a sample-and-hold device, and that the state $z(\cdot)$ evolves continuously.
This approach is essentially different from some of the existing techniques adopted in for example, \cite{AndrNadr13, Libe03TAC}, where the state of the observer/controller is updated in discrete manner whenever the new measurements are available (periodically).
To implement the controller \eqref{eq:contLin}, equation \eqref{eq:contLina} is integrated over the interval $[t_k,t_{k+1})$, $k \in \N$. The expressions on the right-hand side of \eqref{eq:contLina} need quantized values of the output $y$ and the state $z$ at time $t_k$.
Also, the input $u$ needs measurements of the state $z$ updated at time instants $\tau_j$.

It is well-known that if the pair $(A,B)$ is stabilizable and $(A,C)$ is detectable, and the exact values of the output $y(\cdot)$ and input $u(\cdot)$ are continuously available, then the controller \eqref{eq:contLin} stabilizes the system \eqref{eq:sysLin}.
In essence, we are using the ``emulation approach'' where the controller first estimates the current value of the state in the presence of limited information of the output, and the feedback control law using these estimated states then stabilizes the system despite the limitations imposed due to time-sampling and space-quantization.
In the sections to follow, we will specify the algorithms for the sampling and quantization of outputs and inputs that will result in asymptotic stabilization of the closed-loop system~\eqref{eq:sysLin}-\eqref{eq:contLin}.
To derive results in that setup, the following basic assumptions are essential:
\begin{hyp}
\item \label{ass:stab}The pair $(A,B)$ is stabilizable, and hence for every symmetric positive definite matrix $Q_c$ (denoted as $Q_c > 0$) there exists a matrix $P_c > 0$ such that
\begin{equation}\label{eq:lmiCon}
(A+BK)^\top P_c + P_c(A+BK) = -Q_c.
\end{equation}
\item \label{ass:det}The pair $(A,C)$ is observable, so that for every $Q_o>0$,  there exists a matrix $P_o> 0$ such that
\begin{equation}\label{eq:lmiOut}
(A-LC)^\top P_o + P_o(A-LC) = -Q_o.
\end{equation}
\end{hyp}

For the sake of simplicity, let us also introduce the following assumption\footnote{The relaxation to consider systems with imaginary eigenvalues will be addressed in Section~\ref{sec:relax}.}
\begin{hyp}
\item \label{ass:eig} All the eigenvalues of the matrix $A$, denoted $\lambda_i(A)$, $1 \le i \le n$, are real.
\end{hyp}

\section{Output Processing Unit}\label{sec:output}

This section provides the algorithms for transmission of output from the plant to the controller, which leads to our first result stated as Theorem~\ref{thm:quantOut}.
In Section~\ref{sec:sampOut}, we give an algorithm to generate sampling times when the output must be sent to the controller.
In Section~\ref{sec:quantOut}, we then provide an encoding scheme for the sampled output measurements.
It is then shown that the proposed algorithms for sampling and quantization result in the state estimation error dynamics being asymptotically stable and that the proposed sampling strategy introduces a uniform lower bound on inter-sampling times.

\subsection{Event-triggered Sampling of the Output}\label{sec:sampOut}

Let $\tilde x:= x - z$ denote the state estimation error, then \eqref{eq:sysLin} and \eqref{eq:contLin} result in the following equations for error dynamics:
\begin{equation}\label{eq:errSol}
\begin{aligned}
\dot {\tilde x}(t) &= A\tilde x(t) - Lq_{\nu_k}(\tilde y(t_k)), \quad t\in[t_k,t_{k+1}), \\
\tilde y(t) & = C \tilde x(t).
\end{aligned}
\end{equation}

\subsubsection{Construction of a dead-beat-like estimator}

Our goal for this section is to find a relation between $\tilde x(t)$ and certain past-sampled output values of $\tilde y$ measured over the interval $[t_0,t)$ using equation~\eqref{eq:errSol}.
Basically, in \eqref{eq:errSol}, $q_{\nu_k}(\tilde y(t_k))$ acts as a known input, and using $\tilde y$ as the output, it is possible under the observability assumption to reconstruct $\tilde x$ as a function of (sufficiently many) sampled values of $\tilde y$.
Towards this end, let
\begin{equation}
\psi (s_1,s_2,s_3) := Ce^{As_1}\int_{s_2}^{s_3} e^{-As} L \, ds, \!\!\! \quad s_1 \le s_2 \le s_3,
\end{equation}
so that $\psi$ takes values in $\R^{p \times p}$. Now, for $t > t_k > t_{k-1} > \cdots > t_{k-\eta-1}\ge t_0$, define the lower-triangular matrix $\Psi_{k,\eta}(t)$ as 
\begin{equation*}
\Psi_{k,\eta}(t) :=
\begin{bmatrix} 
\psi(t_k,t_k,t) & 0 & 0 \ \cdots \\
\psi(t_{k-1},t_k,t) & \psi(t_{k-1},t_{k-1},t_k) & 0 \  \cdots \\
\psi(t_{k-2},t_k,t) & \psi(t_{k-2},t_{k-1},t_k) & \!\! \psi(t_{k-2},t_{k-2},t_{k-1})\, \cdots  \\
\vdots & \vdots & \vdots 
\end{bmatrix}
\end{equation*}
where $\eta \in \N$ is some strictly positive integer, and to keep the notation short, it is noted that $\Psi_{k,\eta}(t)$ depends on $(t,t_k,t_{k-1}, \dots, t_{k-\eta-1})$.
Next, introduce the following notation:
\begin{equation}
N_{k,\eta}(t) := \begin{bmatrix}Ce^{-A(t-t_k)} \\ Ce^{-A(t-t_{k-1})} \\ \vdots \\Ce^{-A(t-t_{k-\eta+1})}\end{bmatrix}, \quad
\widetilde Y_{k,\eta} := \begin{pmatrix} \tilde y(t_k) \\ \tilde y(t_{k-1}) \\ \vdots \\ \tilde y (t_{k-\eta+1})\end{pmatrix}, \quad
q_{\nu_{k,\eta}}(\widetilde Y_{k,\eta}) := \begin{pmatrix} q_{\nu_k}(\tilde y(t_k)) \\ q_{\nu_{k-1}}(\tilde y(t_{k-1}) ) \\ \vdots \\ q_{\nu_{k-\eta+1}}(\tilde y (t_{k-\eta+1}))\end{pmatrix}.
\end{equation}
The following lemma is a straightforward consequence of the variation of constants formula applied to system~\eqref{eq:errSol} and for completeness, the proof is given in \ref{app:proofs}.
\begin{lemma}\label{lem:relYx}
For any $t > t_k > t_{k-1} > \cdots > t_{k-\eta+1}$, it holds that
\begin{equation}\label{eq:errYk}
N_{k,\eta}(t)\tilde x (t) = \widetilde Y_{k,\eta} - \Psi_{k,\eta}(t) q_{\nu_{k,\eta}}(\widetilde Y_{k,\eta}).
\end{equation}
\end{lemma}

We will use this important relation~\eqref{eq:errYk} to define the sampling times for output measurements, but before proceeding to that, let us recall a few well-known results.
An important requirement in our analysis of minimum inter-sampling time is the invertibility of the matrix $N_{k,\eta}(t)$, for each $t > t_{\eta-1}$, and is achieved due to following lemma:

\begin{lemma}[{\cite[Chapter~6]{Chen99}}]
Under Assumption~\ref{ass:eig}, if we choose $\eta$ to be the observability index of the pair $(A,C)$, then the matrix $N_{k,\eta}(t)$ is left-invertible for each $t > t_{k-\eta+1}$.
\end{lemma}

Thus, $\eta$ could be interpreted as the number of samples required for observability of the discretized system~\eqref{eq:errSol}.
The above result in general holds only under Assumption~\ref{ass:eig}, but we will address the relaxation of this assumption in Section~\ref{sec:relax}. Basically, what changes is that, we have to increase the number of samples over a compact interval to maintain the invertibility of $N_{k,\eta}(t)$. Also, $\eta$ is no longer constant in that case, and depends on the size of the time-interval over which the samples of $\tilde y$ are considered.

\subsubsection{Sampling Algorithm}

To define sampling times at which the output must be sent to the controller, we first chose $t_0 < t_1 < \cdots < t_{\eta-1}$ arbitrarily, and let
\[
f(t,\widetilde Y_{k,\eta}) = \widetilde Y_{k,\eta} - \Psi_{k,\eta}(t) q_{\nu_{k,\eta}}(\widetilde Y_{k,\eta}), \qquad k \ge \eta - 1.
\]
The output sampling time instants $t_{k+1}$ are defined as
\begin{subequations}\label{eq:syReal}
\begin{align}
t_{k+1}^{\text{event}} & := \inf \{t: \|N_{k,\eta}(t)\| \cdot |\tilde y(t) - \tilde y(t_k)| \ge \alpha \cdot |f(t, \widetilde Y_{k,\eta})| \wedge t > t_k\} \label{eq:syReala} \\
t_{k+1}^\text{pers} &:= \inf \{t: t-t_k \ge T\} \label{eq:syRealb} \\
t_{k+1} &:= \min\{t_{k+1}^{\text{event}}, t_{k+1}^\text{pers}\} \label{eq:syRealc}
\end{align}
\end{subequations}
where $\displaystyle \alpha:=\varepsilon_o\frac{\lambda_{\min}(Q_o)}{2\, \|P_oL\|}$, for some $\varepsilon_o \in (0,1)$ and $T>0$ is some prespecified constant, which can be arbitrarily large, but finite. The sampling rule \eqref{eq:syReal} guarantees that the output measurements are transmitted persistently to the controller.

\begin{rem}[Heuristics for Output Sampling]
In comparison to the existing literature on event triggered sampling strategies, expression~\eqref{eq:syReal} is new and has not appeared elsewhere.
One could motivate this particular choice of sampling rule with following arguments:
\begin{itemize}
\item It is noted that the proposed sampling rule \eqref{eq:syReala} guarantees $|\tilde y(t) -\tilde y(t_k)| \le \alpha \cdot \|N_{k,\eta}(t)\|^{-1} |N_{k,\eta}(t)\tilde x(t)| \le \alpha |\tilde x(t)|$.
If we ignore quantization for the time-being, and use $V_o=\tilde x^\top P_o \tilde x$ as the candidate Lyapunov function for the error dynamics \eqref{eq:errSol}, it can then be shown that the resulting state estimation error indeed converges to zero, see equations \eqref{eq:errDynManip} and \eqref{eq:boundVest} for mathematical expressions.

\item In the usual event driven schemes, for example the ones given in \cite{Tabu07}, the statement~\eqref{eq:syRealb} automatically holds, but in our case, due to quantization error in the output measurements, it may be that the state estimation only converges to a certain ball around origin. It may happen that the event condition described in \eqref{eq:syReala} does not become true after that, and the convergence to the origin may not be achieved.
Precisely speaking, the usual event-triggering condition using the full state information is not true as long as:
\[
|x(t) - x(t_k)| < \alpha' |x(t)|, \qquad t > t_k
\]
for some constant $\alpha' > 0$. Using the reverse triangle inequality and the fact that $x(\cdot)$ is decreasing exponentially between two updates, it follows that
\[
|x(t_k)| < (1+\alpha') |x(t)| < (1+\alpha') c_1 e^{-c_2 (t-t_k)} |x(t_k)|, \quad t > t_k.
\]
The expression on the right-hand side converges to zero, and hence for each $x(t_k) \neq 0$, the foregoing inequality is violated after some finite time.
This may not hold for event-triggering condition of the form \eqref{eq:syReala}.
To avoid that, we make sure that an updated output measurement is sent after some time to avoid such situation.
\end{itemize}
\end{rem}

\subsection{Output Quantization}\label{sec:quantOut}

We now define an encoding strategy that is used to transmit $\tilde y(t_k)$ at each time instant $t_k$ using a string of finite length.
The quantization model we use is adopted from \cite{Libe03Aut}, which is a dynamic one.
For output measurement, we assume that the quantizer has a scalable parameter $\nu$ and has the form:
\begin{equation}\label{eq:yQuant}
q_\nu(y) = \nu q^y\left(\frac{y}{\nu}\right)
\end{equation}
where $q^y(\cdot)$ denotes a finite-level quantizer with sensitivity parameterized by $\Delta_y$ and range $R_y$, that is, {\bf if $|y| \le R_y$, then $|q^y(y) -y| \le \Delta_y$}.
This way, the range of the quantizer $q_\nu(\cdot)$ is $R_y\nu$ and the sensitivity is $\Delta_y \nu$.
Increasing $\nu$ would mean that we are increasing the range of the quantizer with large quantization errors and decreasing $\nu$ corresponds to finer quantization with smaller range.
It will be assumed that the quantizer is centred around the origin, that is, $q^y(y) = 0$ if $|y| < \Delta_y$.

We remark that the stability of dynamical systems with quantized measurements has been studied extensively over the past decade, see the survey \cite{NairFagn07} and references therein. 
Stability using quantized and periodically-sampled output measurements (without asymptotic observer) was considered in \cite{Libe03TAC}, and using asymptotic observers (without sampling) was considered in \cite{Libe03Aut}. To the best of our knowledge, stability where both the inputs and outputs are quantized and aperiodically sampled, has not been treated.
In doing so, we find that the quantization parameter for control input depends upon the parameter chosen for output quantization, in order to capture the growth of the estimated state, and that there is a trade-off between how fast we sample and how precisely we quantize due to the choice of respective parameters.


We will now specify an update rule for the parameter $\nu$, so that the state estimation error $\tilde x$ converges to zero.
First, we pick $\nu_0, \cdots, \nu_{\eta-1}$ to be arbitrary. It is assumed that $\nu_\eta$ is chosen such that\footnote{
If $\tilde x(t_0)$ is known to belong to a known bounded set, then $\nu_\eta$ satisfying~\eqref{eq:nuInit} is computed from calculating an upper bound on $|\tilde x(t_\eta)|$ using the differential equation~\eqref{eq:errSol}. One can also use the relation \eqref{eq:errYk} to obtain an upper bound on $\tilde x$ at certain time, or use the strategy proposed in \cite{Libe03Aut, ShimTanw12} to get a bound on state estimation error. To keep the notation simple we have used the same index for sampling times and quantization parameter.}
$\tilde x(t_\eta)$ is contained in an ellipsoid:
\begin{equation}\label{eq:nuInit}
V_o(\tilde x(t_\eta)) \le \frac{\lambda_{\min}(P_o) R_y^2}{\|C\|^2} \, \nu_\eta^2,
\end{equation}
then $|\tilde x(t_\eta)| \le \frac{R_y}{\|C\|} \, \nu_0$ and $|C \tilde x(t_\eta)| \le R_y \nu_0$.
Suppose that we have chosen $\nu_k$ such that~\eqref{eq:nuInit} holds for $\tilde x(t_k)$, for some $k \ge \N$. We will now specify $\nu_{k+1}$ such that~\eqref{eq:nuInit} holds for $\tilde x(t_{k+1})$, for all $t_k$, $k\in \N$,  and at the same time $\lim_{k\rightarrow \infty} \nu_k = 0$.

Since the controller receives the quantized measurements only, the observer takes the following form over the interval $t \in [t_k,t_{k+1})$:
\begin{equation}\label{eq:contLinqo}
\dot z(t) = Az(t) + B u(t) +Lq_{\nu_k}(y(t_k)-Cz(t_k)).  \\
\end{equation}
The dynamics of the state estimation error for the interval $[t_k,t_{k+1})$ are:
\begin{subequations}\label{eq:errDynManip}
\begin{align}
\dot {\tilde x}(t) &= A\tilde x(t) - Lq_{\nu_k}(\tilde y(t_k))\label{eq:errDynManipa}\\
&= A\tilde x(t) - L\tilde y(t_k) - Lq_{\nu_k}(\tilde y(t_k)) + L \tilde y(t_k)\label{eq:errDynManipb}\\
& = (A-LC)\tilde x(t) +L(\tilde y(t) - \tilde y(t_k))  - \nu_k L\left(q^y\left(\frac{\tilde y(t_k)}{\nu_k}\right)-\frac{\tilde y(t_k)}{\nu_k}\right). \label{eq:errDynManipc}
\end{align}
\end{subequations}

Pick $V_o(\tilde x) = \tilde x^\top P_o \tilde x$ as the Lyapunov function, and we see that the measurement update rule \eqref{eq:syReal} leads to the following bound for $t \in [t_k, t_{k+1})$, $k \ge \eta$:
\begin{align}
\dot V_o(\tilde x(t)) & \le -\lambda_{\min}(Q_o) |\tilde x(t)|^2 + 2 \|P_oL\| \, \vert \tilde y(t) - \tilde y(t_k) \vert  \, \vert \tilde x(t) \vert + 2 \nu_k \, \Delta_y \, \|P_oL\| \, |\tilde x(t)|\\
& \le -(1-\varepsilon_o)\lambda_{\min}(Q_o) |\tilde x(t)|^2 + 2 \nu_k \, \Delta_y \, \|P_oL\| \, |\tilde x(t)|, \label{eq:boundVest}
\end{align}
where we used the fact that our output sampling algorithm gives $\vert \tilde y(t) - \tilde y(t_k) \vert \le \alpha \vert \tilde x(t) \vert$.
It then follows that, within two measurement updates, the error converges to a ball parameterized by $\nu_k$.
In particular, for some $0 < \xi_o < \frac{(1-\varepsilon_0)\lambda_{\min}(Q_o)}{\lambda_{\max}(P_o)}$, if we let
\begin{equation}\label{eq:chiOut}
\chi_o := \frac{2 \, \|P_oL\|}{(1-\varepsilon_o)\lambda_{\min}(Q_o)- \xi_o \lambda_{\max}(P_o)} 
\end{equation}
then $|\tilde x(t)| \ge \chi_o \Delta_y \nu_k$ implies that
\[
\dot V_o(\tilde x(t)) \le -\xi_o V_o (\tilde x(t)).
\]
Thus, it follows that, for $t \in [t_k,t_{k+1})$, $k \ge \eta$:
\[
V_o(\tilde x(t)) \le \max \{\lambda_{\max}(P_o)\chi_o^2 \Delta_y^2\nu_k^2, e^{-\xi_o(t-t_k)}V_o(\tilde x(t_k))\}.
\]
For each $k \ge 0$, letting
\[
\Theta_{k+1} := \max \left\{\lambda_{\max}(P_o)\chi_o^2 \Delta_y^2\nu_k^2, e^{-\xi_o(t_{k+1}-t_k)}\frac{\lambda_{\min}(P_o)R_y^2}{\|C\|^2} \nu_k^2\right\},
\]
it follows that $V_o(\tilde x(t_{k+1})) \le \Theta_{k+1}$, for $k\ge \eta$.
If we now pick $\nu_{k+1}$, $k \ge \eta$, as follows:
\begin{equation}\label{eq:upnu}
\boxed{
\nu_{k+1} := \frac{\|C\|}{R_y} \sqrt{\frac{\Theta_{k+1}}{\lambda_{\min}(P_o)}}
}
\end{equation}
then it is guaranteed that $|\tilde x(t_{k+1})| \le \frac{R_y}{\|C\|} \nu_{k+1}$ and $|C \tilde x(t_{k+1})| \le R_y \nu_{k+1}$.


\subsection{Convergence of State Estimation Error}\label{sec:outConv}

Based on the sampling strategy developed in Section~\ref{sec:sampOut}, and the quantization algorithm given in Section~\ref{sec:quantOut}, we are now ready to state our first main result which relates to the convergence of the state estimation error to the origin.

\begin{thm}\label{thm:quantOut}
Assume that the information transmitted from the plant $\cP$ to the controller $\cC$, given by $q_{\nu_k}(\tilde y(t_k))$, $k \ge 1$, is such that
\begin{itemize}
\item The sampling instants $t_k$, $k \ge \eta$, are determined by the relation \eqref{eq:syReal}.
\item For the dynamic quantizer \eqref{eq:yQuant}, the parameter $\nu_\eta$ is chosen to satisfy \eqref{eq:nuInit} and $\nu_k$, $k > \eta$, is updated according to \eqref{eq:upnu}. Moreover, the number of quantization levels determined by $R_y$ and $\Delta_y$ is such that
\begin{equation}\label{eq:bitsOut}
\boxed{
\frac{\Delta_y}{R_y} = \sqrt{\frac{\lambda_{\min}(P_o)}{\lambda_{\max}(P_o)}} \cdot \frac{\overline \rho}{\chi_o\, \|C\|}
}
\end{equation}
for some $\overline\rho \in (0,1)$.
\end{itemize}
Then the following statements holds:
\begin{itemize}
\item There is a minimum dwell-time between two sampling times, that is, there exists $t_D >0$ such that $t_{k+1} - t_k \ge t_D$, for each $k \ge \eta$.
\item The error dynamics \eqref{eq:errSol} are asymptotically stable.
\end{itemize}
\end{thm}

\begin{rem}[Trade-off between sampling and quantization]
In order to maximize the inter-sampling time, one may choose $\alpha$ in the expression \eqref{eq:syReala} by selecting a large value of $\eps_o$.
However, the value $\eps_o$ closer to 1 results in the larger value of $\chi_o$ introduced in \eqref{eq:chiOut}.
From expression \eqref{eq:bitsOut}, it is seen that a large value of $\chi_o$ results in a large number of quantization levels to guarantee asymptotic convergence.
Hence, slower sampling requires greater number of quantization levels and leads to faster convergence of the parameter $\nu$, and vice versa.
In order to minimize both the sampling rate and the quantization levels, that is, increase $\alpha$ without increasing $\chi$, one way is to maximize the ratio $\frac{\lambda_{\min}(Q_o)}{\|P_oL\|}$ by selecting $L$ and $Q_o$ appropriately.
\end{rem}

\begin{proofof}{Theorem~\ref{thm:quantOut}}
Assuming that the first statement holds, let us show that the second statement follows directly from the construction given in Section~\ref{sec:quantOut}.

{\em Step~1--Asymptotic Stability:}
For $k \ge \eta$, using the definition of $\nu_k$ and the condition~\eqref{eq:bitsOut}, it follows that $\nu_{k+1} = \rho \nu_k$, where $\rho:=\max\{\overline \rho, e^{-\xi_o t_D/2}\} < 1$, and we let $t_D$ denote the lower bound on the inter-sampling time for the output\footnote{
We emphasize that we don't need to calculate $t_D$ for the design of quantizer.
}, that is, $t_{k+1} - t_k \ge t_D$.
To show attractivity, we recall from the analysis of Section~\ref{sec:quantOut} that our choice of quantization parameter and sampling strategy guarantees that $V_o(\tilde x(t)) \le \Theta_k$, for each $t \in [t_{k-1},t_k]$, $k \ge \eta$.
From~\eqref{eq:upnu}, it follows that $\Theta_k= \frac{R_y^2\lambda_{\min}(P_o)}{\|C\|^2} \nu_k^2$ and thus, $\Theta_k$ converges to zero as $k$ increases because $\nu_k$ decreases monotonically and $\lim_{k\rightarrow \infty} \nu_k = 0$. Consequently, $\lim_{t \rightarrow \infty}V_o(\tilde x(t)) = 0$.

To show stability in the sense of Lyapunov for error dynamics~\eqref{eq:errSol}, we first state the following lemma whose proof is given in \ref{app:proofs}.

\begin{lemma}\label{lem:lyapStabErr}
The following statement holds for the error dynamics ~\eqref{eq:errSol}:
for every $\bar \eps> 0$, there exists $\delta_0(\bar \eps)>0$, and $\bar k (\bar \eps) \in \N$, such that $\forall \, \tilde x(t_0) \in \R^n$ with $|\tilde x(t_0)| < \delta_0$, the resulting trajectory satisfies $q_{\nu_k}(\tilde y(t_k)) = 0$ for $k < \bar k$, and $\nu_k < \bar \eps$ for $k \ge \bar k$.
\end{lemma}

From equation~\eqref{eq:boundVest}, for every $\eps > 0$, we can find an $\bar \eps >0$ such that if $\nu _{k} < \bar \eps$, then $\dot V_o(\tilde x (t)) < 0$ for $|\tilde x(t)| \ge \eps$.
Thus, if $\nu_{\bar k} < \bar \eps$ and $|\tilde x (t_{\bar k})| < \eps$, then $|\tilde x(t)|< \eps$, $\forall t \ge t_{\bar k}$ because $\nu$ is decreasing monotonically. Let $\delta_0(\bar \eps)$ and $\bar k (\bar \eps)$ be as given in Lemma~\ref{lem:lyapStabErr}, and choose $|\tilde x(t_0)|< \delta$, where $\delta:=\min \{\delta_0,\eps/\|e^{-A\bar k T}\| \}$. Since $q_{\nu_k}(\tilde y(t_k)) = 0$, $k < \bar k$, the solution of \eqref{eq:errSol} results in $|\tilde x(t)| < \eps$, $t \in [t_0,t_{\bar k}]$, and hence the stability follows.

\medskip

{\em Step~2--Dwell-time bound for output updates:} To derive the lower bound on inter-sampling times for output measurements, we let $v(t):=\|N_{k,\eta}(t)\|(\tilde y(t) - \tilde y(t_k))$ and $w(t):=f(t,\widetilde Y_{k,\eta}) = N_{k,\eta}(t) \tilde x(t)$, $\in[t_k,t_{k+1})$, and we now compute the minimum time it takes for $\frac{|v|}{|w|}$ to go from $0$ to $\alpha$.
To achieve that, we first derive a bound on the derivative of $\frac{|v|}{|w|}$.
Towards that end, we first recall that for any real-valued functions $v$ and $w$
\begin{align}
\frac{d}{dt} \frac{|v(t)|}{|w(t)|} & = \frac{v^\top \dot v}{|v| \cdot |w|} - \frac{w^\top \dot w}{|w|^2} \, \frac{|v|}{|w|} \le \frac{|\dot v|}{|w|} + \frac{|\dot w|}{|w|} \, \frac{|v|}{|w|}.\label{eq:ddtBound}
\end{align}
Using $N_{k,\eta}^\dagger (t)$ to denote the left pseudo-inverse of $N_{k,\eta}(t)$ (also see \ref{app:facts}), we now compute a bound on both terms appearing on the right-hand side.
\begin{align}
 \frac{|\dot v(t)|}{|w(t)|} & \le \frac {\|N_{k,\eta}(t)\|\|\dot N_{k,\eta}(t)\|\,|\tilde y(t) - \tilde y(t_k)|}{|N_{k,\eta}(t)\tilde x(t)|} + \frac{\|N_{k,\eta}(t)\| |\dot {\tilde y}(t)|}{|N_{k,\eta}(t) \tilde x(t)|} \notag \\
 & \le \|\dot N_{k,\eta}(t)\| \frac {|v(t)|}{|w(t)|} + \|N_{k,\eta}(t)\| \frac {\|CA\| \cdot |\tilde x(t)|}{|N_{k,\eta}(t)\tilde x(t)|} + \|N_{k,\eta}(t)\| \|CL\| \frac{|q_{\nu_k}(\tilde y(t_k))|}{|N_{k,\eta}(t)\tilde x(t)|} \notag\\
 & \le \|\dot N_{k,\eta}(t)\| \frac {|v(t)|}{|w(t)|} + \|N_{k,\eta}(t)\| \, \|N_{k,\eta}^\dagger(t) \| \|CA\| + \|N_{k,\eta}(t)\| \, \|CL\| \, \|N_{k,\eta}^\dagger(t)\| \frac{|q_{\nu_k}(\tilde y(t_k))|}{|\tilde x(t)|}.
 \label{eq:preBndvw}
\end{align}
To find a bound on $\frac{|q_{\nu_k}(\tilde y(t_k))|}{|\tilde x(t)|}$, we recall that $q_{\nu_k}(\tilde y(t_k)) = 0$, if $|\tilde y(t_k)| < \nu_k \Delta_y$.
Otherwise, if $|\tilde y(t_k)| \ge \nu_k \Delta_y$, we have
\begin{align}
\frac{|q_{\nu_k}(\tilde y(t_k))|}{|\tilde x(t)|} & \le \frac{\nu_k\left|q^y\left(\frac{\tilde y(t_k)}{\nu_k}\right)-\frac{\tilde y(t_k)}{\nu_k}\right| + |\tilde y(t_k)|}{|\tilde x(t)|} \notag\\
&\le \frac{\nu_k \Delta_y + |\tilde y(t_k)|}{|\tilde x(t)|} \notag \\
& \le \frac{2\,|\tilde y(t_k)|}{|\tilde x(t)|} \notag \\
& \le 2\, \frac{|\tilde y(t) - \tilde y(t_k)| + |\tilde y(t)|}{|\tilde x(t)|} \le 2 \left (\frac{|v(t)|}{|w(t)|} + \|C\|\right). \label{eq:bndqyx}
\end{align}
Substituting this relation in \eqref{eq:preBndvw}, we get
\begin{equation}\label{eq:firstTermBndvw}
 \frac{|\dot v(t)|}{|w(t)|} \le \|N_{k,\eta}(t)\| \, \|N_{k,\eta}^\dagger(t) \| \, (\|CA\| +2 \, \|CL\| \, \|C\|) + (\|\dot N_{k,\eta}(t)\| +2 \|CL\| \, \|N_{k,\eta}^\dagger(t)\|)  \frac {|v(t)|}{|w(t)|}.
\end{equation}
To find an upper bound for the second term on the right-hand side of \eqref{eq:ddtBound}, we have
\begin{align}
 \frac{|\dot w(t)|}{|w(t)|} & = \frac{|\dot N_{k,\eta}(t) \tilde x(t) + N_{k,\eta}(t) \dot {\tilde x}(t)|}{|N_{k,\eta}(t)\tilde x(t)|} \notag\\
 &= \frac{|\dot N_{k,\eta}(t) \tilde x(t) + N_{k,\eta}(t) A \tilde x(t)+ N_{k,\eta}(t)L q_{\nu_k}(\tilde y(t_k))}{|N_{k,\eta}(t) \tilde x(t)|} \notag\\
 & \le \|N_{k,\eta}^\dagger(t) \| \, \| \dot N_{k,\eta}(t) \| + \|N_{k,\eta}^\dagger(t) \| \, \| N_{k,\eta}(t) A\| + 2 \|N_{k,\eta}^\dagger(t) \| \, \|L\| \left(\frac {|v(t)|}{|w(t)|} + \|N_{k,\eta}(t)\|\|C\| \right). \label{eq:wdot/w}
\end{align}
Substituting \eqref{eq:firstTermBndvw} and \eqref{eq:wdot/w} in \eqref{eq:ddtBound}, we can write
\begin{equation}\label{eq:outTempBoundRatio}
\frac{d}{dt} \frac{|v(t)|}{|w(t)|} \le g_1(t) + g_2(t) \frac{|v(t)|}{|w(t)|} + g_3(t) \frac{|v(t)|^2}{|w(t)|^2}
\end{equation}
where $g_1(t):= \|N_{k,\eta}(t)\| \,\|N_{k,\eta}^\dagger(t) \| (\|CA\| + 2 \|CL\| \|C\| )$, $g_2(t):= \|\dot N_{k,\eta}(t)\| + \|N_{k,\eta}^\dagger(t) \| \big( 2\|CL\| + \|\dot N_{k,\eta}(t) \| + \|N_{k,\eta}(t) A \| +2 \|N_{k,\eta}(t)\| \| L \| \|C\| \big)$, and $g_3(t):= 2 \|N_{k,\eta}^\dagger (t)\| \,\|L\|$.
Using the facts \ref{fact:Nbound}, \ref{fact:Ninv}, \ref{fact:Ndot} listed in \ref{app:facts}, and letting $\overline \sigma: = \max\{\sigma_1 + \sigma, \sigma_1+ \sigma_2\}$, we can upper bound the right-hand side of \eqref{eq:outTempBoundRatio} as follows:
\begin{equation}\label{eq:outSampdBound}
\frac{d}{dt} \frac{|v(t)|}{|w(t)|} \le a_{1,k} \,e^{\overline \sigma(t -t_k)}\left(\frac{|v(t)|^2}{|w(t)|^2} + a_2 \frac{|v(t)|}{|w(t)|} + a_3\right)
\end{equation}
where $a_{1,k} := c_1 e^{\overline \sigma(t_k-t_{k-\eta+1})} \|L\|$, $a_2:= (1+c)\|C\| + (c_1+c_2+c\|A\|)/\|L\|$ and $a_3:=c\|C\|(\|C\|+\|A\|/\|L\|)$.
To derive an expression for minimum inter-sampling times, we state the following lemma, the proof of which is given in \ref{app:proofs}.
\begin{lemma}\label{lem:solBound}
For some positive scalars, $a_1,a_2,a_3, \overline \sigma$, if $X: [t_k,\infty) \rightarrow \R$ satisfies
\[
\frac{dX}{dt} \le a_1 e^{\overline \sigma(t-t_k)}(X^2 + a_2 X + a_3 )
\]
with initial condition $X(t_k) = 0$, then there exists $\tilde t_D>0$ such that 
\[
X(t) \le \frac{ra_2}{2} \tan(a_4(e^{\overline \sigma(t-t_k)}-1)), \qquad t-t_k < \tilde t_D,
\]
for some large enough positive integer $r$, and the positive scalar $a_4:=\frac{ra_2a_1}{2\overline\sigma}$.
\end{lemma}
Applying Lemma~\ref{lem:solBound} to \eqref{eq:outSampdBound}, there exists $\tilde t_{D,k} > 0$ such that
\[
\frac{|v(t)|}{|w(t)|} \le \frac{ra_2}{2} \tan(a_{4,k}(e^{\overline \sigma(t-t_k)}-1)), \quad t-t_k < \tilde t_{D,k},
\]
for some $r$ large enough and $a_{4,k} : = \frac{r a_2 a_{1,k}}{2\overline \sigma}$. Now, it is easily verified that $\frac{|v(t)|}{|w(t)|} \ge \alpha$ for $t > t_k$, only if
\begin{equation}\label{eq:seqOut}
 (t-t_k) \ge \frac{1}{\overline \sigma}\log \left(1+\frac{1}{a_{4,k}}\arctan\left(\frac{2\,\alpha}{r a_2}\right)\right).
\end{equation}
The sampling strategy~\eqref{eq:syReal} guarantees that $t_{k+1} - t_k \le T$, for all $k \ge \eta$, so that the bound $a_{4,k} \le \frac{ra_2c_1\|L\|}{2\overline \sigma} e^{\overline \sigma(\eta-1)T}$ holds for all $k \ge \eta$ and using this bound in \eqref{eq:seqOut} gives a uniform lower bound for inter-sampling times.
\end{proofof}
\begin{rem}
Once we arrive at \eqref{eq:seqOut} towards the end of Theorem~1, it can be shown that a sequence of numbers lower bounded by the right-hand side of \eqref{eq:seqOut} does not converge (without using any upper bound on $t_{k+1} - t_k$), and hence the corresponding series diverges showing that there is no accumulation point for sampling times. However, to derive an expression for uniform lower bound on inter-sampling times, and guarantee persistent sampling, we do assume that $t_{k+1}-t_k \le T$.
\end{rem}

\subsection{Relaxing Assumption~\ref{ass:eig}}\label{sec:relax}

We now want to consider the case when the matrix $A$ doesn't necessarily have real eigenvalues. In that case, Lemma~\ref{lem:relYx} does not hold in general. In this section, we argue that the invertibility of the matrix can still be guaranteed if we work with sufficiently large number of samples over a fixed interval. We recall the following result from \cite{WangLi11}:

\begin{lemma}
Let $\omega:=\max_{1\le i,j \le n} \{\im(\lambda_i(A) - \lambda_j(A))\}$. If
\begin{equation}
\eta > 2(n-1) + \frac{T_s}{2\pi} \omega
\end{equation}
then the matrix $\col(Ce^{As_1}, Ce^{As_2}, \cdots, Ce^{As_\eta})$ is left-invertible for all $s_1,s_2, \dots,s_\eta \in [0,T_s]$.
\end{lemma}

To use this lemma for our problem setup, we first fix some integer $\eta^* > 2(n-1)$.
Having picked $t_0 < t_1 < \dots < t_{\eta^*-1} $ arbitrarily, we choose the next sampling time $t_{k+1}:= \min \{t_{k+1}^{\text{event}}, t_{k+1}^{\text{sample}}\}$ recursively, where
\[
t_{k+1}^{\text{sample}}  := \inf\{t: t- t_{k-\eta^*+1} > \min\left\{ \frac{2\pi}{\omega} \left(\eta^* - 2(n-1)\right), \eta^*T \right\} \ \wedge \ t > t_k \}
\]
and $t_{k+1}^{\text{event}} $ is defined as
\[
t_{k+1}^{\text{event}}  := \inf \{t: \|N_{k,\eta}(t)\| \cdot |\tilde y(t) - \tilde y(t_k)| \ge \alpha \cdot |f(t, \widetilde Y_{k,\eta})| \wedge t > t_k\}
\]
where once again, we pick $\displaystyle \alpha:=\varepsilon_o\frac{\lambda_{\min}(Q_o)}{2\, \|P_oL\|}$, for some $\varepsilon_o \in (0,1)$, and $T>0$ is some constant.
In the case, when $\omega = 0$, which is the case if all eigenvalues of $A$ are real, then $t_{k+1}^{\text{sample}}$ matches the definition given in \eqref{eq:syRealb}.
The stability analysis would not be affected by this modified sampling strategy, and the existence of a lower bound could be shown by modifying the proof given earlier.

\section{Input Processing Unit}\label{sec:input}

We can tailor the ideas introduced in the previous section to derive a sampling algorithm and a quantization strategy for the control input.

\subsection{Sampling Algorithms for Inputs}\label{sec:sampIn}

Let $\tau_0 = t_\eta$, and choose the control input $u$, so that $u(t) = 0$, for $t \in [t_0, \tau_0)$, and $u(\tau_0) = Kz(\tau_0) = Kz(t_{\eta})$, and the next update is performed at $\tau_{j+1}$, which, for $j \ge 0$, is defined as follows:
\begin{subequations}\label{eq:su}
\begin{align}
\tau_{j+1}^{\text{event}} &:=\inf \{t:|Kz(t) - Kz(\tau_j)| \ge \left(\beta_c |z(t)| + \beta_o\frac{R_y}{\|C\|}\nu(t)\right) \ \wedge \ t > \tau_j\},\label{eq:sua}\\
\tau_{j+1}^{\text{pers}} &:= \inf\{t: t-\tau_j \ge T\}, \\
\tau_{j+1} &:= \min\{\tau_j^{\text{event}}, \tau_j^{\text{pers}}\},
\end{align}
\end{subequations}
where $\displaystyle \beta_c:=\varepsilon_c\frac{\lambda_{\min}(Q_c)}{2 \|P_cB\|}$, and $\beta_o:=\tilde \beta \beta_c \|C\|$ for some $\varepsilon_c \in (0,1)$, and $\tilde \beta >0$.

Note that the term $\nu(\cdot)$ is only piecewise constant and does not vary continuously with time. In case there is a time $t_k > \tau_j$ such that $|Kz(t_k) - Kz(\tau_j)| < \beta_c |z(t_k)| + \beta_o R_y\nu(t_k^-)$, and due to sudden change in the value of $\nu$ at time $t_k$, it happens that $|Kz(t_k) - Kz(\tau_j)| \ge \beta_c |z(t_k)| + \beta_o R_y\nu(t_k^+)$, then in that case we assume that $\tau_{j+1} = t_k$, and hence the control input is updated instantaneously without any delay.

\subsection{Input Quantization}\label{sec:quantIn}
In our setup, the control input cannot be transmitted to the plant with exact precision and only $q_{\mu_j}(Kz(\tau_j))$, $j \in \N$ is transmitted to the plant.
The quantization model used for control inputs is similar to the one adopted for outputs, that is,
\[
q_\mu(u) = \mu\,q^u\left(\frac{u}{\mu}\right)
\]
where $\mu$ denotes the scaling parameter, and $q^u$ is a finite-level quantizer whose range is denoted by $R_u$, and the sensitivity by $\Delta_u$.
We specify an update rule for the parameter $\mu_j$ associated with the input quantizer such that the resulting closed-loop system is still globally asymptotically stable.
In order to do that, we choose $z(\tau_0)$ such that
\[
V_c(z(\tau_0)) \le \frac{\lambda_{\min}(P_c)R_u^2}{\|K\|^2}\mu_0^2.
\]
With quantized inputs and outputs, the dynamical system \eqref{eq:contLin} is thus written as:
\begin{align*}
\dot z (t) &= (A+BK)z(t) + B(u(t) - Kz(t)) + L (q_{\nu_k} (y(t_k)) - Cz (t_k)), \quad t,t_k \in [\tau_j,\tau_{j+1}) \\
&=(A+BK)z(t) +BK (z(\tau_j) - z(t))  + \mu_j B\left(q\left(\frac{Kz(\tau_j)}{\mu_j}\right) - \frac{Kz(\tau_j)}{\mu_j}\right) + L (q_{\nu_k} (y(t_k) - Cz (t_k))). 
\end{align*}
With $V_c(z) = z^\top P_c z$ as the Lyapunov function, and the control update rule \eqref{eq:su}, we observe that
\begin{multline}
\dot V_c (z(t)) \le -(1-\varepsilon_c) \lambda_{\min}(Q_c) |z(t)|^2 + |z(t)| (\tilde \beta \eps_c \lambda_{\min}(Q_c) R_y\nu_{k^*(t)} + 2\|P_cB\| \Delta_u \mu_j) \\+ 2 \, |z(t)|\, \| P_c L\| (|\tilde y(t_{k^*(t)}) |+\nu_{k^*(t)}\Delta_y)
\end{multline}
where 
\[
k^*(t):= \max \{k \in \N : t_k \le t\}.
\]
From our output quantization scheme, we have that $|\tilde y(t_k)| = |C\tilde x(t_k)| \le R_y \nu_k$, for all $k \in \N$, and $\nu_{k^*(\tau_j)} \ge \nu_{k^*(t)}$, for $t \ge \tau_j$.
For a fixed $0 < \xi_c < \frac{(1-\varepsilon_c)\lambda_{\min}(Q_c)}{\lambda_{\max}(P_c)}$, we introduce the constants
\begin{equation}\label{eq:chiIn}
\chi_c := \frac{2\|P_cB\|}{(1-\varepsilon_c)\lambda_{\min}(Q_c)-\xi_c \lambda_{\max}(P_c)}
\end{equation}
and
\[
\zeta_1 := \frac{\tilde \beta \eps_c\lambda_{\min}(Q_c)+ 2\|P_cL\|}{(1-\varepsilon_c)\lambda_{\min}(Q_c)-\xi_c \lambda_{\max}(P_c)}, \quad 
\zeta_2 := \frac{2\|P_cL\|}{(1-\varepsilon_c)\lambda_{\min}(Q_c)-\xi_c \lambda_{\max}(P_c)}.
\]
It is noted that if $|z(t)| \ge \overline \chi_j := \chi_c \Delta_u \mu_j + (\zeta_1R_y + \zeta_2\Delta_y)\nu_{k^*(\tau_j)}$, then
\[
\dot V_c(z(t)) \le -\xi_c V_c (z(t)).
\]
Assuming that $z(\tau_j)$ is contained in an ellipsoid defined as:
\[
V_c(z(\tau_j)) \le \frac{\lambda_{\min}(P_c)R_u^2}{\|K\|^2} \mu_j^2,
\]
we let
\[
\Theta^u_{j+1} := \max \left\{\lambda_{\max}(P_c)\overline \chi_j^2, e^{-\xi_c(\tau_{j+1}-\tau_j)}\frac{\lambda_{\min}(P_c)R_u^2}{\|K\|^2} \mu_j^2 \right\}.
\]
Choose $\mu_{j+1}$ such that
\begin{equation}\label{eq:upmu}
\boxed{
\mu_{j+1}^2 = \frac{\|K\|^2 \Theta^u_{j+1}}{R_u^2 \lambda_{\min}(P_c)}
}
\end{equation}
then it is guaranteed that $| z(\tau_{j+1})| \le \frac{R_u}{\|K\|} \mu_{j+1}$, and $|Kz(\tau_{j+1})| \le R_u \mu_{j+1}$.

\begin{rem}
In order to implement the quantization algorithm for the control inputs, it must be noted that the parameter $\mu$ actually depends on the parameter $\nu$ used for the quantization of $\tilde y$.
This is done because the evolution of the controller state $z$ actually depends upon the quantized values of $\tilde y$, and to determine the region that contains the state $z$ at current time instant, we use the knowledge of how large $\tilde y$ is, which is indeed captured by the most recent value of $\nu$.
\end{rem}

\subsection{Asymptotic Stability of the Plant}

Based on the sampling routine and quantization algorithm proposed for the control input, we now show that the resulting plant dynamics are indeed asymptotically stable.
More formally, we have the following result:

\begin{thm}
Suppose that the output sent to the controller satisfies the design criteria of Theorem~\ref{thm:quantOut}, and the control input $u(t) = q_{\mu_j}(Kz(\tau_j))$, $t \in [\tau_j,\tau_{j+1})$ is such that
\begin{itemize}
\item The input sampling instants $\tau_j$, $j \ge 1$ are determined by~\eqref{eq:su}.
\item The parameter $\mu_j$ for the dynamic quantization of the input is updated according to~\eqref{eq:upmu}, and the quantization levels of $q^u$ determined by $R_u$ and $\Delta_u$ satisfy the following relation:
\begin{equation}\label{eq:bitsCon}
\boxed{
\frac{\Delta_u}{R_u} = \sqrt {\frac{\lambda_{\min}(P_c)}{\lambda_{\max}(P_c)}} \frac{\overline \rho_u}{\chi_c \|K\| }
}
\end{equation}
for some $\overline \rho_u \in (0,1)$.
\end{itemize}
Then, the following statements hold:
\begin{itemize}
\item There is a minimum dwell-time between two control input updates, that is, there exists $\tau_D > 0$ such that $\tau_{j+1} - \tau_j \ge \tau_D$, for every $j \ge 0$.
\item The closed-loop system \eqref{eq:sysLin}-\eqref{eq:contLin} is asymptotically stable.
\end{itemize}
\end{thm}

\begin{proof}
We have already demonstrated that the error dynamics~\eqref{eq:errSol} are asymptotically stable, and it would suffice to show that the dynamics of \eqref{eq:contLin} satisfy the same property.

{\em Asymptotic stability of system~\eqref{eq:contLin}:}
Assuming that the first statement holds, that is, there is $\tau_D>0$ such that $\tau_{j+1} - \tau_j \ge \tau_D$, for all $j \in \N$ (the existence of which will be proved later in the proof),
then for $\rho_u:= \max\{\overline \rho_u, e^{-\xi_c\tau_D/2}\} < 1$, we have
\[
\mu_{j+1} = \rho_u \mu_j + \rho_y \nu_{k^*(\tau_j)},
\]
where $\rho_y:=\frac{\|K\|}{R_u}\sqrt {\frac{\lambda_{\max}(P_c)}{\lambda_{\min}(P_c)}}(\zeta_1R_y+\zeta_2\Delta_y)$.
Since ${k^*(\tau_j)}$ goes to infinity as $j$ goes to infinity, and $\nu_k$ is a decreasing sequence, it follows that $\mu_j$ is uniformly bounded and eventually converges to $0$.

Stability in the sense of Lyapunov follows the same argument as used in the proof of Theorem~\ref{thm:quantOut} and Lemma~\ref{lem:lyapStabErr}. One can choose $|\tilde x(t_0)|$ and $|z(t_0)|$ sufficiently small such that the quantized values of $u$ and $\tilde y$ remain zero until the corresponding quantization parameters are small enough. After which the negative value of $\dot V_c(\cdot)$ guarantees that $z(\cdot)$ remains inside the prescribed ball.

{\em Dwell-time between two input updates:}
In the definition of sampling times $\tau_j$ given in \eqref{eq:su}, it is noted that $|\tilde x(t)| \le \frac{R_y}{\|C\|}\nu(t)$, for all $t \ge t_\eta$.
Thus, a lower bound on $\tau_{j+1}-\tau_j$ is the time it takes for $|K z(t) - K z(\tau_j)|$ to go from $0$ to $\beta (|z(t)| + |\tilde x(t)|)$, where $\beta:=\max\{\beta_c,\beta_o\}$.
We let $v(t)$ denote $K(z(t) - z(\tau_j))$ and let $w(t)$ denote $|z(t)| + |\tilde x(t)|$ in \eqref{eq:ddtBound}, $t \in [\tau_j,\tau_{j+1})$, to observe that
\begin{equation}\label{eq:vwRatio1}
\frac{|\dot v(t)|}{|w(t)|} \le \frac{\|KA\| \, |z(t)| + \|KB\| \,| q_{\mu_j}(Kz(\tau_j))| + \|KL\| |q_{\nu_k}(\tilde y(t_k))|}{|z(t)| + |\tilde x(t)|}.
\end{equation}
To find a bound on $\frac{|q_{\mu_j}(Kz(\tau_j))|}{|z(t)|}$, we use the same technique as in the derivation of \eqref{eq:bndqyx}.
Recall that $q_{\mu_j}(Kz(\tau_j)) = 0$, if $|Kz(\tau_j)| < \mu_j \Delta_u$;
otherwise, if $|Kz(\tau_j)| \ge \mu_j \Delta_u$, we have
\begin{align}
\frac{|q_{\mu_j}(Kz(\tau_j))|}{|z(t)|+|\tilde x(t)|} & \le \frac{\mu_j\left|q^u\left(\frac{Kz(\tau_j)}{\mu_j}\right)-\frac{Kz(\tau_j)}{\mu_j}\right| + |Kz(\tau_j)|}{|z(t)|+|\tilde x(t)|} \notag\\
&\le \frac{\mu_j \Delta_u + |Kz(\tau_j)|}{|z(t)|+|\tilde x(t)|} \notag \\
& \le \frac{2\,|Kz(\tau_j)|}{|z(t)|+|\tilde x(t)|} \notag \\
& \le 2\, \frac{|Kz(t) - Kz(\tau_j)| + |Kz(t)|}{|z(t)|+|\tilde x(t)|} \le 2 \left (\frac{|v(t)|}{|w(t)|} + \|K\|\right). \label{eq:bndqux}
\end{align}
Similarly, in \eqref{eq:bndqyx}, if we use the fact that $|\tilde y(t) - \tilde y(t_k)| \le \alpha |\tilde x(t)|$ because of the sampling rule \eqref{eq:syReala}, we get
\[
\frac{|q_{\nu_k}(\tilde y(t_k))|}{|\tilde x(t)|} \le 2(\alpha + \|C\|).
\]
Substituting these bounds in \eqref{eq:vwRatio1} results in
\[
\frac{|\dot v(t)|}{|w(t)|} \le b_1 + b_2 \frac{|v(t)|}{|w(t)|}
\]
for some appropriately defined constants $b_1,b_2 > 0$.
Using similar logic, it holds that for some $b_3, b_4 > 0$
\[
\frac{|\dot w(t)|}{|w(t)|} = \frac{|\dot z (t)| + |\dot{\tilde x}(t)|}{|z(t)| + |\tilde x(t)|} \le b_3 + b_4\frac{|v(t)|}{|w(t)|}.
\]
Thus, using the expression~\eqref{eq:ddtBound}, we obtain
\[
\frac{d}{dt} \frac{|v(t)|}{|w(t)|} \le b_5\left(\frac{|v(t)|^2}{|w(t)|^2} + b_6 \frac{|v(t)|}{|w(t)|} + b_7\right).
\]
Once again using Lemma~\ref{lem:solBound} (see Remark~\ref{rem:sig0} in \ref{app:proofs}), we can find a $\tilde \tau_D > 0$ such that
\[
\frac{|v(t)|}{|w(t)|} \le \frac{rb_6}{2} \tan(b_8(t-\tau_j)), \qquad t-\tau_j < \tilde \tau_{D},
\]
for some $r$ large enough and $b_8 : = \frac{r b_6 b_{5}}{2}$. Hence,
\[
\tau_{j+1} - \tau_j \ge \tau_{D} := \frac{1}{b_{8}}\arctan\left(\frac{2\,\beta}{r b_6}\right).
\]
It is seen that the term $\tau_{D}$ is strictly positive and hence the minimum time between two updates of the control values have a uniform lower bound.
\end{proof}


\section{Illustrative example}\label{sec:example}
Consider the system with following matrices:
\[
A:= \begin{bmatrix}1 & 1\\ 0 & 0.5 \end{bmatrix}; \quad B:= \begin{bmatrix}0 \\ 1 \end{bmatrix}; \quad C:= \begin{bmatrix} 1 & 0\end{bmatrix}.
\]
Since the matrix $A$ doesn't have any complex eigenvalues, it suffices to take $\eta = 2$.
For the state estimation part, we choose the output injection gain $L=[4~3]^\top$, $Q_o = [{1\atop 0.5} {0.5 \atop 1}]$, and $\eps_o = 0.75$ which results in $P_o:= [{1.63 \atop -1.47}{-1.47 \atop 1.93}]$ and $\alpha=0.09$.
For quantization of the sampled output, we pick a quantizer $q^y$ which rounds off the real-valued output to the nearest integer, so that $\Delta_y = 1$.
The value of parameter $\xi_o = 0.1 \frac{(1-\varepsilon_0)\lambda_{\min}(Q_o)}{\lambda_{\max}(P_o)} = 0.0038$ results in $\chi_o = 37.54$. Finally by selecting $\bar \rho = 0.975$, it is seen that the number of quantization levels required for convergence of state estimation error are
\[
\frac{R_y}{\Delta_y} \ge 126.4,
\]
that is, we need $\lceil \log_2 (126.4) \rceil = 7$ bit quantizer for the output. It must be recalled that no optimality criterion was placed in obtaining the required number of bits for convergence and it could be reduced for other choices of matrices $L$ and $Q_o$.

For the control input, the feedback gain matrix $K=-[6~ 4.5]$ is chosen with $Q_c = Q_o$, and $\eps_c=0.85$. This results in $P_c = [{2.5 \atop 0.25}{0.25 \atop 0.5}]$, and $\beta_c = 0.38$. For the quantization, we again pick $q^u$ such that its input is rounded of to the nearest integer, so that $\Delta_u = 1$.
The value of parameter $\xi_c = 0.1 \frac{(1-\varepsilon_c)\lambda_{\min}(Q_c)}{\lambda_{\max}(P_c)} = 0.0048$ results in $\chi_c = 9.94$. Finally by selecting $\bar \rho_u = 0.85$, it is seen that the desired value of $R_u = 318$, so we need $\lceil \log_2 (318) \rceil = 9$ bit quantizer for the control input.
The results of the simulation are given in Figure~\ref{fig:simLinCase}.
As expected, the states of the system converge to zero under the proposed algorithm, and the plots in Figures~\ref{fig:outData}, and~\ref{fig:conData} show the sampled and real values of the output and input, respectively.

\begin{figure}[!t]
     \centering
     \subfigure[The plot contains the continuous-time signal $\tilde y$ (continuous curve in blue) with its space-quantized and time-sampled values transmitted to the controller (dots in red color). The quantization parameter $\nu$ converges to zero very quickly, and hence a very little difference is observed between the quantized and real values.]{
          \label{fig:outData}
          \includegraphics[height=1.65in]{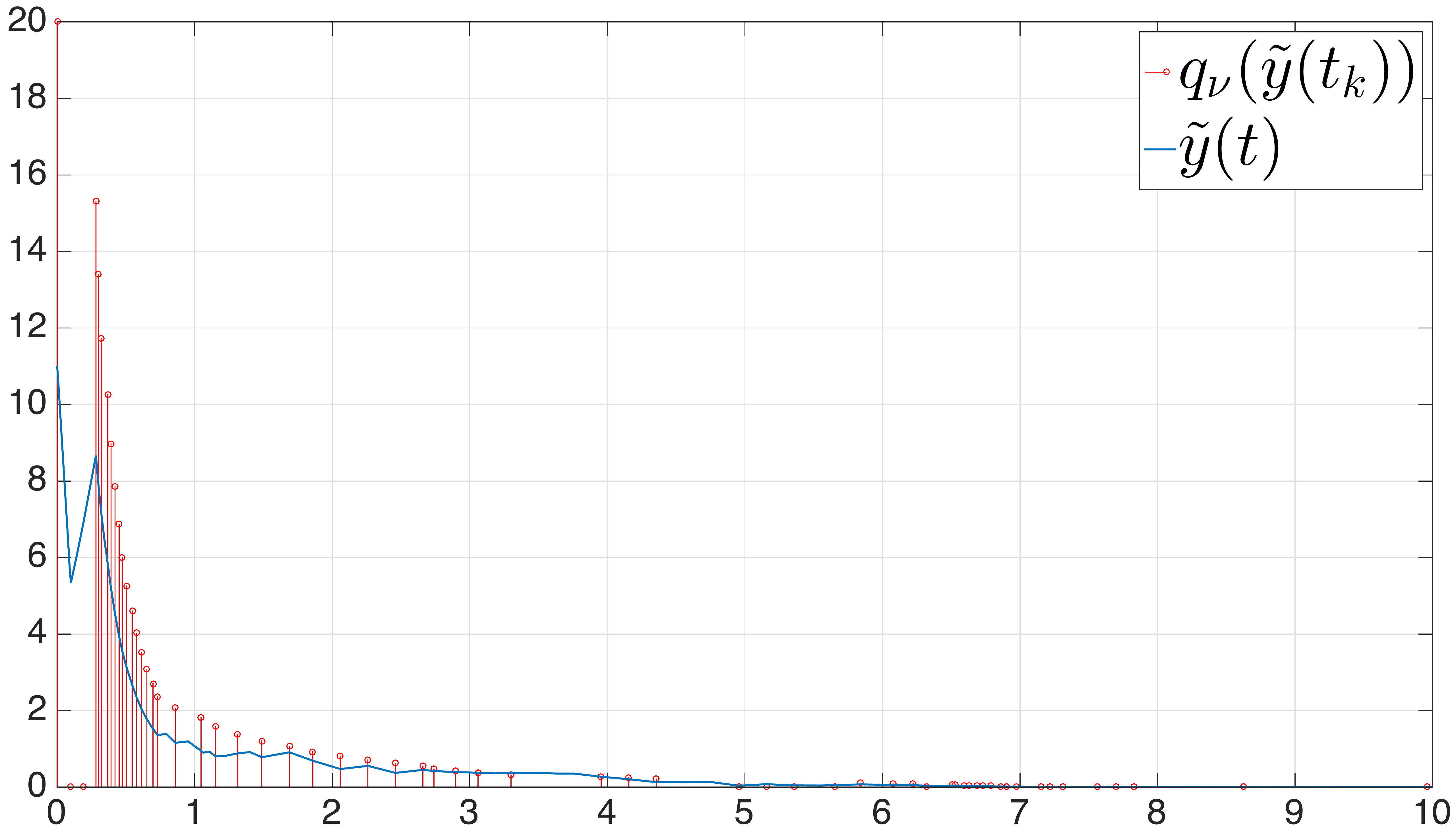}}
     \hspace{0.05in}
     \subfigure[Plot of the continuous-time control input with its space-quantized and time-sampled values transmitted to the plant. Initially, the value of the quantization parameter $\mu$ is very large since it waits for $\tilde y$ to converge, and afterwards, with smaller values of $\mu$, the states start converging to the origin.]{
          \label{fig:conData}
          \includegraphics[height=1.65in]{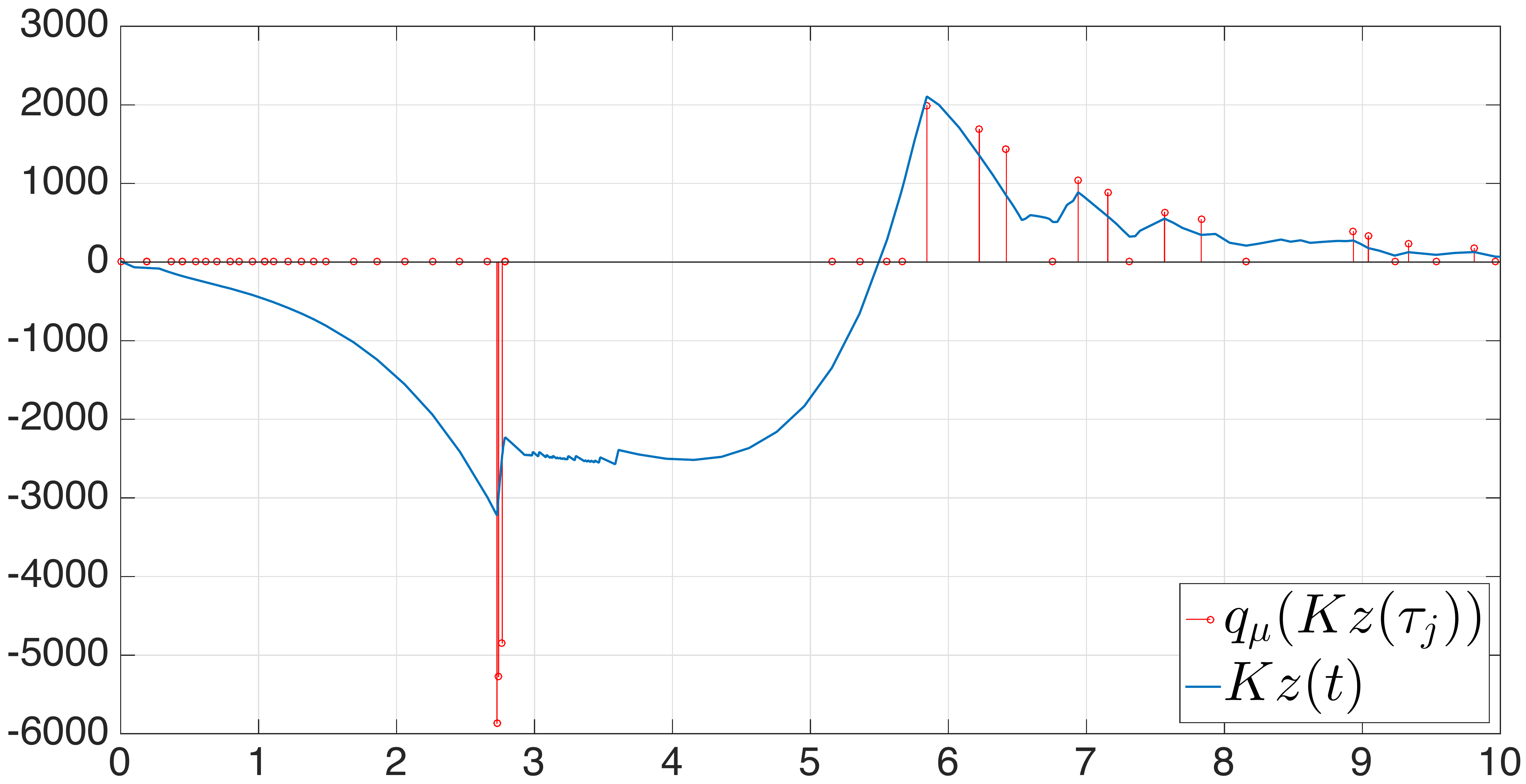}}\\
     \vspace{.05in}
     \caption{Input and output signals resulting from the design of a stabilizing controller with time-sampled and space-quantized measurements.}
     \label{fig:simLinCase}
\end{figure}

%

%

\section{Conclusions}


For the future work, certain issues could be investigated further. The choice of parameters $\alpha$ and $\beta$ directly affect the inter-sampling time observed in the simulations. So it is of interest to maximize them by choosing appropriate feedback and injection gain matrices while keeping the number of required quantization levels to a minimum.
Also, it is possible that the proposed methodology could be extended to nonlinear systems.
For that, some modifications need to be introduced which do not require exact computation of the discretized model of the plant (which is impossible to compute in nonlinear systems). This will also alleviate the need of computing exponentials of matrices in our algorithms.

\appendix

\section{Proofs of Lemmas}\label{app:proofs}

\begin{proofof}{Lemma~\ref{lem:relYx}}
For $t \ge t_k$, we have,
\[
\tilde x(t) = e^{A(t-t_k)} \tilde x(t_k) - \int_{t_k}^t e^{A(t-s)}L \, q_{\nu_k}(\tilde y(t_k)) \, ds
\]
or equivalently,
\begin{equation}\label{eq:xtil1}
e^{-A(t-t_k)} \tilde x(t) = \tilde x(t_k) - \int_{t_k}^t e^{A(t_k-s)} L \, ds \, q_{\nu_k}(\tilde y(t_k)).
\end{equation}
Using similar calculations, one can solve the system~\eqref{eq:errSol} over the interval $[t_{k-1},t_k]$ to obtain:
\begin{equation}\label{eq:xtil2}
e^{-A(t_k-t_{k-1})} \tilde x(t_k) = \tilde x(t_{k-1}) - \int_{t_{k-1}}^{t_k} e^{A(t_{k-1}-s)} L \, ds \, q_{\nu_{k-1}}(\tilde y(t_{k-1})).
\end{equation}
Since the solution of \eqref{eq:errSol} are absolutely continuous, one can substitute the value of $\tilde x(t_{k-1})$ from \eqref{eq:xtil2} in \eqref{eq:xtil1} to obtain
\begin{equation}\label{eq:xtil3}
e^{-A(t-t_{k-1})} \tilde x(t) = \tilde x(t_{k-1}) - \int_{t_k}^{t}e^{A(t_{k-1}-s)} L \, ds \, q_{\nu_k}(\tilde y(t_k)) - \int_{t_{k-1}}^{t_k} e^{A(t_{k-1}-s)} L \, ds \, q_{\nu_{k-1}}(\tilde y(t_{k-1})).
\end{equation}
Equations~\eqref{eq:xtil1} and \eqref{eq:xtil2} show that \eqref{eq:errYk} holds for $\eta=2$. For higher values of $\eta$, one can continue in similar manner and solve \eqref{eq:errSol} on precedent intervals $[t_{k-2},t_{k-1}], \dots, [t_{k-\eta+1},t_{k-\eta+2}]$ to arrive at the desired expression in~\eqref{eq:errYk}.
\end{proofof}

\begin{proofof}{Lemma~\ref{lem:lyapStabErr}}
From the update rule specified for the quantization parameter $\nu$, we have $\nu_{k+1} := \rho \nu_k$, for $k \ge \eta$. Thus, there exists $\bar k$, such that $\nu_k < \bar \eps$, for $k \ge \bar k \ge \eta$.
It remains to show that $q_{\nu_k}(\tilde y(t_k)) = 0$ for sufficiently small $|\tilde x(t_0)|$.
Towards that end, let $a_0:=\|e^{A(T+t_{\eta-1}-t_0)}\|$, $a_\eta:=\|e^{A(\bar k-\eta)T}\|$, and $a:=a_0a_\eta$, so that the solution of $\dot{\tilde x} = A \tilde x$ satisfies
\[
|\tilde x(t)| \le a \, |\tilde x(t_0)|, \quad t\in[t_0,t_{\bar k}).
\]
Choose $\nu_0 = \nu_1 = \dots = \nu_\eta$ large enough, and $|\tilde x(t_0)|$ small enough such that $|\tilde x(t_0)| < \frac{\rho_{\min}^{\bar k}\Delta_y}{a\|C\|}\nu_0 =: \delta_0$, where $\rho_{\min}:=\min \{\overline \rho, e^{-\delta T/2}\}$.
This would imply that, for $0 \le k \le \eta-1$, $t \in [t_{k},t_{k+1})$, $|\tilde x(t)| \le \frac{\rho_{\min}^{\bar k}\Delta_y}{a_\eta\|C\|}\nu_0$. In particular, for $0 \le k \le \eta$, $|\tilde y(t_k)| \le \Delta_y \nu_k$, and therefore $q_{\nu_k}(\tilde y(t_k))=0$.
Since $|\tilde x(t_\eta)| < \frac{\Delta_y}{a_\eta\|C\|} \rho_{\min}^{\bar k} \nu_\eta$ and $\nu_{\bar k} \ge \rho_{\min}^{\bar k} \nu_\eta$, it follows that for $\eta+1 \le k \le \bar k$, $t \in [t_{k-1},t_{k})$, $|\tilde x(t)| \le \frac{\rho_{\min}^{\bar k}\Delta_y}{\|C\|}\nu_\eta \le \frac{\Delta_y}{\|C\|} \nu_k$. Consequently, $q_{\nu_k}(\tilde y(t_k)) = 0$, for each $\eta+1 \le k \le \bar k$.
\end{proofof}

\begin{proofof}{Lemma~\ref{lem:solBound}}
For some positive integer $r$, let $\tilde a_3:= (r^2+1) \frac{a_2^2}{4} $ be such that $\tilde a_3 \ge a_3$. By comparison lemma, $X(t) \le Z(t)$, for each $t \ge t_k$, where $Z(\cdot)$ satisfies the differential equation
\[
\frac{dZ}{dt} = a_1 e^{\overline \sigma(t-t_k)}(Z^2 + a_2 Z + \tilde a_3 )
\]
with initial condition $Z(t_k) = 0$, so that any upper bound on $Z(\cdot)$ acts as a bound on $X(\cdot)$. It is noted that
\[
\int_{0}^{Z(t)}\frac{dZ}{Z^2 + a_2 Z + \tilde a_3} = a_1 \int_{t_k}^{t}e^{\overline \sigma(s-t_k)}ds
\]
or equivalently,
\[
\int_{0}^{Z(t)}\frac{dZ}{\left(Z +\frac{a_2}{2}\right)^2 + \left(\tilde a_3-\frac{a_2^2}{4}\right)} = \frac{a_1}{\overline \sigma} \left(e^{\overline \sigma(t-t_k)} - 1\right) .
\]
Using the definition of $\tilde a_3$, and integrating both sides, we get
\begin{equation*}
\frac{2}{r a_2}\arctan\left(\frac{2}{ra_2}\left(Z +\frac{a_2}{2}\right)\right) - \frac{2}{ra_2}\arctan \left(\frac{1}{r}\right) = \frac{a_1}{\overline \sigma} \left(e^{\overline \sigma(t-t_k)} - 1\right)
\end{equation*}
Letting $a_4:= \frac{ra_1a_2}{2\overline \sigma}$, we get
\begin{align*}
Z(t) &= \frac{ra_2}{2}\tan\left(a_4(e^{\overline \sigma(t-t_k)}-1) + \arctan\left(\frac{1}{r}\right)\right)- \frac{a_2}{2} \\
& = \frac{a_2}{2} \left( \frac{r\tan(a_4(e^{\overline \sigma(t-t_k)}-1)) + 1}{1-\frac{\tan(a_4(e^{\overline \sigma(t-t_k)}-1))}{r}} -1\right).
\end{align*}
In the above expression, note that the term $\tan(a_4(e^{\overline \sigma(t-t_k)}-1))$ takes the value $0$ at $t=t_k$ and is increasing monotonically with $t$.
Thus, if $\frac{\tan(a_4(e^{\overline \sigma(t-t_k)}-1))}{r} < 1$, or equivalently, $(t-t_k)< \frac{1}{\overline \sigma}\log \left(1+ \frac{1}{a_4} \arctan r\right) =: \tilde t_D$, then for $t-t_k < \tilde t_D$, we have
\[
\frac{r\tan(a_4(e^{\overline \sigma(t-t_k)}-1)) + 1}{1-\frac{\tan(a_4(e^{\overline \sigma(t-t_k)}-1))}{r}} < r\tan(a_4(e^{\overline \sigma(t-t_k)}-1)) + 1,
\]
and it follows that
\[
Z(t) \le \frac{ra_2}{2} \tan(a_4(e^{\overline \sigma(t-t_k)}-1)), \qquad t-t_k < \tilde t_D.
\]
Noting that $X(t) \le Z(t)$, for all $t \ge t_k$, the desired result follows.
\end{proofof}
\begin{rem}\label{rem:sig0}
In the statement of Lemma~\ref{lem:solBound}, if $\overline \sigma = 0$, then we get
\[
X(t) \le \frac{ra_2}{2} \tan(a_4(t-t_k)), \qquad t-t_k < \tilde t_D
\]
where $a_4:= \frac{ra_1a_2}{2}$ and $\tilde t_D:= \arctan r$.
\end{rem}

\section{Basic Facts}\label{app:facts}

\paragraph*{Facts} We list a series of facts that were used in the proof of Theorem~\ref{thm:quantOut}.
These relate to bounding the norm of matrix $N_{k,\eta}(\cdot)$, the norm of its pseudo-inverse, and the norm of the matrix obtained by differentiating $N_{k,\eta}(\cdot)$.
\begin{facts}
\item \label{fact:Nbound}There exist constants $c,\sigma > 0$, such that $\|N_{k,\eta}(t)\| \le c \, e^{\sigma (t-t_{k-\eta+1})}$, for all $t \ge t_k$.
\item \label{fact:Ninv} If $N_{k,\eta}(t)$ has rank $n$, then there exists a matrix $N_{k,\eta}^\dagger (t)$ such that $N_{k,\eta}^\dagger(t) \cdot N_{k,\eta}(t) = I_{n \times n}$ and for some positive constants $c_1,\sigma_1$, we have
\[
\|N_{k,\eta}^\dagger(t)\| \le c_1 e^{\sigma_1(t-t_{k-\eta+1})}
\]
and consequently,
$
|N_{k,\eta} (t) x| \ge \frac{|x|}{\|N_{k,\eta}^\dagger(t)\|}
$, 
$ \forall \, x \in \R^n$.
\item \label{fact:Ndot} Let $\dot N_{k,\eta}(t)$ denote the matrix obtained by taking the derivative of each element of $N_{k,\eta}(t)$ with respect to time, then there exist $c_2>0, \sigma_2>0$ such that
\[
\|\dot N_{k,\eta}(t) \| \le c_2 e^{\sigma_2 (t-t_{k-\eta+1})}.
\]
\end{facts}

\begin{rem}
There are different ways in which we can compute the constant $c, c_1, c_2$ and $\sigma, \sigma_1, \sigma_2$.
For example, a conservative way to compute, $c$ and $\sigma$ is to observe that
\[
N_{k,\eta}(t) = \bdiag(C, C, \dots, C) \times
\col (e^{-A(t-t_k)}, e^{-A(t-t_{k-1})}, \dots, e^{-A(t-t_{k-\eta+1})})
\]
where the notation $\col(A_1,A_2)$ denotes $\left[ \begin{smallmatrix} A_1 \\ A_2\end{smallmatrix}\right]$. Hence, one can take $c=n \eta \,\sqrt{n\eta p} \, \|C\| \cdot \bar c$ and $\sigma = \bar \sigma$, where $\bar c, \bar \sigma$ are such that $\|e^{-At}\| \le \bar c \, e^{\bar \sigma t}$, for all $t \ge 0$.
\end{rem}


\begin{thebibliography}{10}

\bibitem{AbdePost14}
M.~Abdelrahim, R.~Postoyan, J.~Daafouz and D.~Nesic.
\newblock Stabilization of nonlinear systems using event-triggered output feedback laws.
\newblock In {\em Proc. Int. Symp. Math. Theory of Networks and Systems}, Groningen 2014. 

\bibitem{AndrNadr13}
V.~Andrieu, M.~Nadri, U.~Serres, J.-C.~Vivalda.
\newblock Continuous discrete observer with updated sampling period.
\newblock In {\em Proc.~9th IFAC Symposium on Nonlinear Control Systems}, Toulouse, 2013.

\bibitem{AntaTabu10}
A.~Anta and P.~Tabuada.
\newblock To sample or not to sample: self-triggered control for nonlinear
  systems.
\newblock {\em IEEE Trans. Automatic Control}, 55(9):2030 -- 2042, 2010.

\bibitem{Astr08}
K.~Astrom.
\newblock Event based control.
\newblock In {\em Analysis and Design of Nonlinear Control Systems}, pages 127
  -- 148. Springer, 2008.

\bibitem{Chen99}
C.-T. Chen.
\newblock {\em Linear System Theory and Design}.
\newblock Oxford University Press, Inc., 3rd edition, 1999.

\bibitem{DePeFras13}
C.~{De Persis} and P.~Frasca.
\newblock Robust self-triggered coordination with ternary controllers.
\newblock {\em {IEEE} Trans. Automatic Control}, 58(12):3024 -- 3038, 2013.

\bibitem{DePeSail13}
C.~{De Persis}, R.~Sailer, and F.~Wirth.
\newblock Parsimonious event-triggered distributed control: {A} {Z}eno free
  approach.
\newblock {\em Automatica}, 49(7):2116 -- 2124, 2013.

\bibitem{DonkHeem12}
M.C.F. Donkers and W.P.M.H. Heemels.
\newblock Output-based event-triggered control with guaranteed
  $\cL_{\infty}$-gain and improved and decentralized event-triggering.
\newblock {\em {IEEE} Trans. Automatic Control}, 57(6):1362 -- 1376, 2012.

\bibitem{GarcAnts13}
E.~Garcia and P.J.~Antsaklis.
\newblock Model-based event-triggered control for systems with quantization and time-varying network delays.
\newblock {\em {IEEE} Trans. Automatic Control}, 58(2):422 -- 434, 2013.

\bibitem{Gira14}
A.~Girard.
\newblock Dynamic event generators for event-triggered control systems.
\newblock submitted; Available online: arXiv:1301.2182, 2013.

\bibitem{HeemDonk13a}
W.P.M.H. Heemels, M.C.F. Donkers.
\newblock Model-based periodic event-triggered control for linear systems.
\newblock {\em Automatica}, 49:698 -- 711, 2013.

\bibitem{HeemDonk13b}
W.P.M.H. Heemels, M.C.F. Donkers and A.~Teel.
\newblock Periodic event-triggered control for linear systems.
\newblock {\em IEEE Trans. Automatic Control}, 58(4):847 -- 861, 2013.

\bibitem{HeemJoha12}
W.P.M.H. Heemels, K.H. Johansson, and P.~Tabuada.
\newblock An introduction to event-triggered and self-triggered control.
\newblock In {\em Proc.~51st {IEEE} Conf.~Decision \& Control}, Maui (HI), pp.~3270 -- 3285, 2012.

\bibitem{LehmLunz11}
D.~Lehmann and J.~Lunze.
\newblock Event-based output-feedback control.
\newblock In {\em Proc.~19th Mediterranean Conf. on Control and Automation},
  pages 982 -- 987, 2011.

\bibitem{Libe03Aut}
D.~Liberzon.
\newblock Hybrid feedback stabilization of of systems with quantized signals.
\newblock {\em Automatica}, 39(9):1543 -- 1554, 2003.

\bibitem{Libe03TAC}
D.~Liberzon.
\newblock On stabilization of linear systems with limited information.
\newblock {\em IEEE Trans. Automatic Control}, 48(2):304 -- 307, 2003.

\bibitem{LunzLehm10}
J.~Lunze and D.~Lehmann.
\newblock A state-feedback approach to event-based control.
\newblock {\em Automatica}, pages 211 -- 215, 2010.

\bibitem{MarcDura13}
N.~Marchand, S.~Durand, and J.F.G. Castellanos.
\newblock A general formula for event-based stabilization of nonlinear systems.
\newblock {\em {IEEE} Trans. Automatic Control}, 58(5):1332 -- 1337, 2013.

\bibitem{MazoAnta10}
M.~{Mazo Jr.}, A.~Anta, and P.~Tabuada.
\newblock An {ISS} self-triggered implementation of linear controllers.
\newblock {\em Automatica}, 46(8):1310 -- 1314, 2010.

\bibitem{MazoCao11}
M.~{Mazo Jr.} and M.~Cao.
\newblock Decentralized event-triggered control with asynchronous updates.
\newblock In {\em Proc. Joint 50th {IEEE} Conf.~Decision \& Control and
  European Control Conf.}, Orlando (FL), pp.~ 2547 -- 2552, 2011.

\bibitem{NairFagn07}
G.~Nair, F.~Fagnani, S.~Zampieri, and R.J. Evans.
\newblock Feedback control under data rate constraints: {A}n overview.
\newblock {\em Proceedings of the {IEEE}}, 95(1):108 -- 137, 2007.

\bibitem{PostAnta11}
R.~Postoyan, A.~Anta, D.~Nesic, and P.~Tabuada.
\newblock A unifying {L}yapunov-based framework for the event triggered control
  of nonlinear systems.
\newblock In {\em Proc. Joint 50th {IEEE} Conf.~Decision \& Control and
  European Control Conf.}, Orlando (FL), pp.~2559 -- 2564, 2011.

\bibitem{SeurPrie11}
A.~Seuret, C.~Prieur, and N.~Marchand.
\newblock Stability of nonlinear systems by means of event-triggered sampling algorithms.
\newblock {\em IMA J. of Math. Control and Information}, 31(3): 415 -- 433, 2014.
  
\bibitem{ShimTanw12}
H.~Shim, A.~Tanwani and Z.~Ping.
\newblock Back-and-forth operation of state observers and norm 
estimation of estimation error.
\newblock In {\em Proc.~51st {IEEE} Conf. on Decision and Control}, Maui(HI), pp.~3221
  -- 3226, 2012.

\bibitem{Tabu07}
P.~Tabuada.
\newblock Event-triggered real-time scheduling of stabilizing control tasks.
\newblock {\em {IEEE} Trans. Automatic Control}, 52(9):1680 -- 1685, 2007.

\bibitem{TallChop12}
P.~Tallapragada and N.~Chopra.
\newblock Event-triggered dynamic output feedback control for {LTI} systems.
\newblock In {\em Proc.~51st {IEEE} Conf. on Decision and Control}, pages 6597
  -- 6602, 2012.
  
\bibitem{TanwTeel15}
A.~Tanwani, A. Teel and C. Prieur.
\newblock On Using Norm Estimators for Event-Triggered Control with Dynamic
Output Feedback.
\newblock In {\em Proc.~54th {IEEE} Conf. on Decision and Control}, pages 5500 -- 5505, 2012.

\bibitem{WangLi11}
L.Y. Wang, C.~Li, G.G. Yin, L.~Guo, and C.-Z. Xu.
\newblock State observability and observers of linear-time-invariant systems
  under irregular sampling and sensor limitations.
\newblock {\em {IEEE} Trans. Automatic Control}, 56(11):2639 -- 2654, 2011.

\bibitem{WangLemm11}
X.~Wang and M.~Lemmon.
\newblock Event triggering in distributed networked control systems.
\newblock {\em {IEEE} Trans. Automatic Control}, 56(3):586 -- 601, 2011.

\bibitem{YuAnts13}
H.~Yu and P.J.~Antsaklis.
\newblock Event-triggered output feedback control for networked control systems using passivity: {A}chieving $\cL_2$ stability in the presence of communication delays and signal quantization.
\newblock{\em Automatica}, 49(1):30 -- 38, 2013.

\end{thebibliography}
\end{document}